\documentclass[12pt,letterpaper,titlepage]{amsart}
\usepackage{amsmath, amssymb, amsthm, amsfonts,amscd,amsaddr}

\usepackage[usenames]{color}

\theoremstyle{plain}
\newtheorem{theorem}{Theorem}[section]

\newtheorem{cor}[theorem]{Corollary}

\newtheorem{prop}[theorem]{Proposition}
\newtheorem{lemma}[theorem]{Lemma}
\newtheorem{definition}[theorem]{Definition}
\theoremstyle{definition}
\newtheorem{ex}[theorem]{Example}
\newtheorem{rmk}[theorem]{Remark}
\numberwithin{equation}{section}
\newtheorem*{theoremA*}{Theorem A}
\newtheorem*{theoremB*}{Theorem B}
\newtheorem*{theoremm1*}{Theorem A'}
\newtheorem*{theoremC*}{Theorem C}
\newtheorem*{theoremD*}{Theorem D}
\newtheorem*{theoremE*}{Theorem E}
\newtheorem*{theoremF*}{Theorem F}
\newtheorem*{theoremE2*}{Theorem E2}
\newtheorem*{theoremE3*}{Theorem E3}
\newcommand{\bs}{\backslash}

\newcommand{\C}{\mathbb{C}}

\newcommand{\Pc}{\mathcal{P}}

\newcommand{\R}{\mathbb{R}}
\newcommand{\N}{\mathbb{N}}

\newcommand{\Aut}{\operatorname{Aut}}
\newcommand{\Gr}{\operatorname{Gr}}
\newcommand{\Sl}{\operatorname{SL}}
\newcommand{\SO}{\operatorname{SO}}

\newcommand{\GL}{\operatorname{GL}}

\newcommand{\Ad}{\operatorname{Ad}}
\newcommand{\ad}{\operatorname{ad}}
\newcommand{\diag}{\operatorname{diag}}

\newcommand{\Span}{\operatorname{span}}

\newcommand{\rank}{\operatorname{rank}}

\newcommand{\Hfx}{H_f^\times}
\newcommand{\hfx}{\hf_f^\times}

\def\hat{\widehat}
\def\af{\mathfrak{a}}
\def\jf{\mathfrak{j}}
\def\ifr{\mathfrak{i}}

\def\gf{\mathfrak{g}}

\def\cf{\mathfrak{c}}
\def\df{\mathfrak{d}}
\def\hf{\mathfrak{h}}
\def\kf{\mathfrak{k}}
\def\lf{\mathfrak{l}}
\def\mf{\mathfrak{m}}
\def\nf{\mathfrak{n}}

\def\pf{\mathfrak{p}}
\def\qf{\mathfrak{q}}

\def\uf{\mathfrak{u}}

\def\zf{\mathfrak{z}}
\def\la{\langle}
\def\ra{\rangle}
\def\1{{\bf1}}

\def\P{\mathcal{P}}

\def\oline{\overline}

\title[Simple compactifications]
{Simple compactifications and polar decomposition of homogeneous real spherical spaces}

\subjclass[2000]{14M27, 22F30, 22E15}
\keywords{spherical space, polar decomposition}
\begin{document}
\date{May 29, 2014}

\begin{abstract}
Let $Z$ be an algebraic  homogeneous space $Z=G/H$ attached to 
real reductive Lie group $G$.
We assume that $Z$ is real spherical, i.e., minimal parabolic subgroups 
have open orbits on $Z$.  For such spaces we investigate 
their large scale geometry  and provide a polar decomposition. 
This is obtained from the existence of simple compactifications  of $Z$ which 
is established in this paper.
\end{abstract}

\author[Knop]{Friedrich Knop}
\email{friedrich.knop@fau.de}
\address{Department Mathematik, Emmy-Noether-Zentrum\\
FAU Erlangen-N\"urnberg, Cauerstr. 11, 91058 Erlangen, Germany} 
\author[Kr\"otz]{Bernhard Kr\"{o}tz}
\email{bkroetz@math.uni-paderborn.de}
\address{Universit\"at Paderborn, Institut f\"ur Mathematik\\Warburger Stra\ss e 100, 
D-33098 Paderborn, Germany}
\thanks{The second named author was supported by ERC Advanced Investigators Grant HARG 268105}
\author[Sayag]{Eitan Sayag}
\email{eitan.sayag@gmail.com}
\address{Department of Mathematics, Ben Gurion University of the Negev\\P.O.B. 653, Be'er Sheva 84105
, Israel}
\author[Schlichtkrull]{Henrik Schlichtkrull}
\email{schlicht@math.ku.dk}
\address{University of Copenhagen, Department of Mathematics\\Universitetsparken 5, 
DK-2100 Copenhagen \O, Denmark}

\maketitle

\section{Introduction}

Our concern is the large scale geometry of algebraic homogeneous spaces $Z=G/H$
attached to an algebraic  real reductive group $G$. Here $H<G$ is an algebraic subgroup. 

\par One approach towards the large scale geometry of $Z$ is to study the double-coset space 
$K\bs G/H$ for $K<G$ a maximal compact subgroup.   
In interesting cases, for example if $H$ is a symmetric subgroup which is compatible with the choice of 
$K$, i.e., $H\cap K$ is maximal compact in $H$, then 
there is a good answer in terms of the  generalized Cartan decomposition for 
$Z=G/H$:  there is a non-compact  torus $A_q$ with 
Lie algebra orthogonal to $\hf +\kf$ such that $G=KA_qH$. 
Here $\hf$ and $\kf$ denote the Lie algebras of $H$ and $K$.

\par The class of homogeneous spaces $Z$ we consider in this paper are those which are 
called {\it real spherical}, i.e., minimal parabolic subgroups of $P$ admit open orbits.
Symmetric spaces are real spherical and basic properties of symmetric spaces 
have been shown to persist in the larger class of real spherical spaces
(see \cite{Bien},  \cite{KSS}, \cite{KS1}, \cite{KKS}, \cite{KS2}).

\par The objective of this paper is to study the large scale geometry of real spherical spaces. 
In the past we looked at many non-symmetric examples and constructed 
non-compact tori $\af$ such that $G=KAH$ holds true, but could not find a general 
construction scheme, see \cite{KSS}, \cite{DKS}. In contrast to symmetric spaces the tori $\af$ are typically 
not orthogonal to $\kf +\hf$ which makes matters rather complicated.

\par In order to discuss the large scale geometry of spherical spaces it thus seems reasonable 
to weaken the concept of the polar decomposition $G=KAH$ and replace $K$ by a compact subset
$\Omega \subset G$. This approach is motivated by the investigations in 
\cite{SV} for a class of p-adic spherical 
spaces.   

\par By definition a  minimal parabolic subgroup $P<G$ is given by $P=G\cap P_\C$ where $P_\C <G_\C$
is a minimal parabolic of $G_\C$ which is defined over $\R$. 
Here $G_\C$ denotes the complexification of $G$ which is a complex reductive algebraic group.

One main geometric result of this paper then is:

\begin{theorem}\label{mainthm}[Polar Decomposition] Let $Z=G/H$ be a real spherical space
attached to an algebraic real reductive group.  Let 
$P$ be a minimal parabolic subgroup of $G$ such that $PH$ is open. 
Then there is a Levi decomposition $P=MA\ltimes N$ 
such that 
$$G=\Omega A F H $$
for a compact set $\Omega\subset G$ and a finite set $F\subset G$. 
\end{theorem}

\par The main new tool for deriving the polar decomposition 
is the existence of a {\it simple compactification}.
The definition is given in Section \ref{simple c}.
Most importantly, a simple compactification has a unique closed orbit, and
thus the following result is obtained in Section \ref{S:exist simple c}.

\begin{theorem}\label{mainthm2}
Let $Z=G/H$ be a real spherical space and assume  that
$H$ is equal to its normalizer, $H=N_G(H)$.
Let the subgroup 
$J<G$ be defined by $$J=\{g\in G\mid P_\C H_\C g=P_\C H_\C\}.$$
Then there exists a compact subgroup $M_J$ of $J$ such that $J=M_JH$. 
In particular, $J/H$ is compact.

Furthermore, there exists an  irreducible 
rational real representation $V$ of $G$ with $J$-fixed vector $v_J$ such that 
$$ Z_J:=G/J \to \mathbb{P}(V), \ \ gH \mapsto [g\cdot v_J]$$
is an  embedding, and such that the closure of $Z_J$ in $\mathbb{P}(V)$ is a $G$-compactification 
of $Z_J$ with a unique closed $G$-orbit.
\end{theorem}

\par The proof of Theorem \ref{mainthm} follows in Section \ref{pode}.
In Section \ref{section cone} we define the compression (or valuation) cone  
of $Z$. 
We show that the compression cone governs the fine convex geometry near
the closed orbit of a simple compactification and results in a refined polar decomposition,
Theorem \ref{generic RPD}. 
Finally, following Sakellaridis and Venkatesh
\cite{SV} we define a class of real spherical 
spaces which satisfy the wavefront lemma of Eskin-McMullen \cite{EM}. 

\par{\it Acknowledgement:} We thank the anonymous referee for useful suggestions
which led to an improvement of our paper. 

\section{Real spherical spaces}

\subsection{Notation on real spherical spaces}\label{notation}

We will denote Lie groups by upper case Latin letters, e.g $A$, $B$ etc., 
and their Lie algebras by lower case  German  letters, e.g. $\af$, $\mathfrak b$ etc.  

\par Let $G$ be an algebraic  real reductive group by which we understand 
an open subgroup of the real points of a connected complex 
reductive algebraic group $G_\C$.
Further  we let $H<G$ be a closed
subgroup such that there is a  complex algebraic subgroup 
$H_\C <G_\C$ such that $G\cap H_\C=H$. Under these assumptions 
we refer to $Z=G/H$ as a real algebraic homogeneous space.  We set $Z_\C=G_\C/H_\C$ 
and note that there is a natural $G$-equivariant embedding 
$$Z\hookrightarrow Z_\C , \ \ gH \mapsto gH_\C\, .$$
Let us denote by $z_0=H$ the standard base point of $Z$.
We denote by $\C[G_\C]$ the ring of regular functions on $G_\C$.

\par Denote by $\P$ the variety of all minimal parabolics. Then
$\P\simeq G/P$ for any given $P\in\P$.
In this paper we will assume that $Z$ is {\it real spherical}, 
i.e.,~some, and hence all, $P\in\P$ admit an open orbit on $Z$.

\subsection{Embeddings}\label{Embeddings} 
It follows from a theorem of Chevalley that every real algebraic homogeneous space $Z=G/H$ admits a $G$-equivariant embedding 
into the projective space of a rational $G$-module $V$, i.e.,~there is  a vector 
$0\neq v_H\in V$ such that $H$ is the stabilizer of the line $[v_H]\in\mathbb{P}(V)$.
Then the map 
\begin{equation}\label{projective embedding}
 Z =G/H \to \mathbb{P}(V), \ \ gH \mapsto [\pi(g)v_H]
\end{equation}
is a $G$-equivariant embedding. 

It can sometimes be useful to reduce matters to the quasi-affine situation 
of an embedding into $V$. This is achieved with the following standard trick. 
For every algebraic character $\chi: H \to \R^\times$ we set
\begin{equation}\label{reduction quasi-affine}
G_1 := G\times \R^\times,\qquad H_{1,\chi}:=\{ (h, \chi^{-1}(h))\in G_1\mid h\in H\},
\end{equation}
then $Z_{1,\chi}= G_1/H_{1,\chi}$ is a real spherical space.
In particular, if $V$ and $v_H$ are as above,
and if $\chi: H \to \R^\times$ is given by $\pi(h) v_H = \chi(h) v_H$, then
$$(g,t)H_{1,\chi}\mapsto t\,\pi(g)v_H, \qquad
Z_{1,\chi}:= G_1/H_{1,\chi} \to V$$
is an embedding.

\subsection{The local structure theorem}\label{lst}

Let $P\in\P$ be such that $PH$ is open in $G$, that 
is $\gf=\pf +\hf$.

\par If $\lf$ is a real reductive Lie algebra then we denote by $\lf_{\rm n}$, reps. $\lf_{\rm c}$  the union of the 
non-compact, resp. compact,  simple ideals of $\lf$. Note that 
$$\lf = \zf(\lf) \oplus \lf_{\rm n} \oplus \lf_{\rm c}$$
is a direct sum of reductive Lie algebras.

\par According to the local structure theorem of \cite[Thm.~2.2]{KKS} there exists a 
parabolic subgroup $Q\supset P$ 
with Levi decomposition $Q=LU$ such that:

\begin{itemize}
\item $Q\cdot z_0 = P\cdot z_0$ and
\item   $L_{\rm n} < Q\cap H < L$ with $L_n =\la \exp \lf_{\rm n}\ra$.
\end{itemize}

Hence on the level of Lie algebras we have 
\begin{equation}\label{deco} \lf_{\rm n}\subset \qf \cap \hf \subset \lf\, .\end{equation}
A parabolic subgroup $Q$ with the properties listed above is said to be $Z$-{\it adapted}. It is shown 
in \cite[Thm.~2.7]{KKS} that only one parabolic subgroup
containing $P$ is $Z$-adapted .

\par We let $K_L A_L N_L=L$ be an Iwasawa decomposition of $L$ and set 
$A:=A_L$. This way, we obtain a Levi decomposition $P=MAN= MA \ltimes N$ where $M=Z_{K_L}(A)$
and $N=N_LU$. 

\par We let 
$$\zf(\lf)= \zf(\lf)_{np} \oplus \zf(\lf)_{cp}$$
be the decomposition into compact and non-compact part, i.e., 
$\zf(\lf)_{cp}$ is the Lie algebra of the maximal compact 
subgroup of the abelian Lie group $Z(L)$ and $\zf(\lf)_{np} $ is its orthocomplement. 

Let $\df:=\zf(\lf) + \lf_{\rm c}$.
As there is no algebraic homomorphism of a non-compact torus 
into a compact group we obtain that every algebraic subalgebra $\cf$ of $\df$
decomposes
$\cf= [\cf\cap\zf(\lf)_{np} ]\oplus [\cf\cap (\zf(\lf)_{cp} +\lf_c)]$.
As $\hf$ is algebraic we thus get  
$$\hf \cap \lf = \lf_{\rm n} + \af_h +\mf_h $$
with $\af_h\subset\zf(\lf)_{np}$  the image of 
$\hf \cap \lf$  under the orthogonal projection $\lf\to \zf(\lf)_{np}$ and likewise for 
$\mf_h < \zf(\lf)_{cp}+\lf _{\rm c}$.
Let $\af_Z\subset \zf(\lf)_{np}$ be the orthogonal complement of $\af_h$ and $\mf_Z$ 
the orthogonal complement to $\mf_h $ in $\lf_c +\zf(\lf)_c$.
Accordingly the following direct sum holds
\begin{equation}\label{deco2} \gf =\hf \oplus \af_Z \oplus \mf_Z \oplus \uf\, .\end{equation}

\subsection{The open $P$-orbits in $Z=G/H$}\label{open orbits}

Let $P\in\P$ be such that $PH$ is open in $G$ and let $Q\supset P$ be $Z$-adapted.
Our goal in this subsection is to describe all open $P$-orbits on $Z=G/H$. 

Recall that the local structure theorem asserts that $PH=QH$ and that there is an algebraic diffeomorphism

\begin{equation} \label{LSTglob} U \times S \to P\cdot z_0\end{equation} 
where $S=L\cdot z_0$ is a homogeneous space for the group $D:= L/L_{\rm n}$.
As $L$ is reductive we can and will assume that $D<L$. 
Note that the Lie algebra $\df = \zf(\lf) +\lf_{\rm c}$ is compact and contained 
in $\af +\mf$ (recall that a Lie algebra 
is called compact if it is isomorphic to the Lie algebra of a compact Lie group).

\begin{lemma}\label{open orbits lemma} 
$P_\C \cdot z_0 \cap Z$ is the union of the open $P$-orbits in $Z$.
\end{lemma}
\begin{proof}  
We first note that 
$P_\C  \cdot z_0$ is Zariski open, hence dense and thus the unique 
open $P_\C$-orbit in $Z_\C$.  Let now $g\in G$ and $z=g\cdot z_0\in Z$.
Then $\Ad(g)^{-1}\pf + \hf=\gf$ if and only if
$\Ad(g)^{-1} \pf_\C +\hf_\C =\gf_\C$, and hence
$P\cdot z$ is open in $Z$ if and only if $P_\C \cdot z$ is open in $Z_\C$.
The lemma follows immediately.
\end{proof}

The local structure theorem was obtained through the use of a $P$-semi-invariant regular function 
$f$ on $G$ (see also Section \ref{SCII}
where we review this in more detail). If we view $f$ as a regular 
function on $G_\C$ we obtain the complex version of the local structure theorem:   
with $Q_\C$ the Zariski closure of $Q$ and $U_\C:= \exp(\uf_\C)$  we obtain 
a parametrization of the open $P_\C$-orbit $P_\C\cdot z_0\subset Z_\C$: 
\begin{equation} \label{LSTglob2}U_\C \times S_\C  \to P_\C \cdot z_0\, .\end{equation}
The slice $S_\C$ is described as follows: With $L_\C<Q_\C$ the Levi part with $L_\C \supset L$ and 
$L_{n,\C}=\la \exp (\lf_{n,\C})\ra$
we obtain that $S_\C=L_\C\cdot z_0$ is a homogeneous space for the group $D_\C:= L_\C / L_{n,\C}$.

\par Note that the number of open $P$-orbits in $Z$ is finite.  
In view of (\ref{LSTglob2}) we have
$(P_\C\cdot z_0)(\R)= U \times S_\C(\R)$.  Hence all open $P$-orbits 
are given by 
\begin{equation} \label{allorbits}  P t_1\cdot z_0, \ldots, Pt_m \cdot z_0\end{equation}
where $t_j \in \exp(i \df)$ (see also (\ref{tj})).  
It is no loss of generality to assume that $t_1=\1$.
In particular we find $e_j\in G$ and $h_j \in H_\C$ such that $t_j= e_j h_j$. 
Set 
\begin{equation}\label{finite set}  F:= \{e_1, \ldots, e_m\}\end{equation} 
and note that 
\begin{equation}
\label{fs2} P g H \subset G \ \hbox{is open for all $g\in F$}\, .\end{equation}

\subsection{Explicit structure of the slice $S$}

As $D$ is an algebraic group we have 

$$D = Z(L)_{np} \times D_c $$
with $D_c$ a compact subgroup with Lie algebra $\zf(\lf)_c + \lf_c$
and $Z(L)_{np}:= \exp \zf(\lf)_{np}$. With $C<D$ the stabilizer of $z_0$ in 
$D$ we obtain likewise that 

$$C = A_h \times M_h$$
with $M_h$ a compact subgroup of $D_c$ with Lie algebra $\mf_h$. As $Z(L)_{np} = A_Z \times A_h$ we conclude that 
$$ S = D/C \simeq A_Z \times D_c/ M_h\, .$$
Set $M_Z:= D_c/M_h$.
In particular, in (\ref{allorbits}) one can arrange that
\begin{equation}\label{tj}
t_j\in T_Z:=\exp(i\af_Z),\quad(j=1,\dots,m).\end{equation}

\section{Simple compactifications}\label{simple c}

In the sequel we use the term $G$-space for a topological space endowed with a continuous $G$-action.

\par By a {\it compactification} of $Z=G/H$ we understand a compact $G$-space $\hat Z$ 
such that 
\begin{itemize}
\item $\hat Z\supset Z$ as $G$-space.
\item $Z$ is open dense in $\hat Z$.
\end{itemize}

Compactifications of real spherical spaces exist: Recall from (\ref{projective embedding})
the $G$-equivariant embedding in $\mathbb{P}(V)$.  
\par According to \cite{KKS} the closure $\hat Z$  of $Z$ in $\mathbb{P}(V)$ 
has a finite orbit decomposition. However, the orbit structure of such an embedding can be quite complicated, in particular 
it can happen that there are closed $G$-orbits of different orbit type. 

\begin{ex} We consider $G=\Sl(2,\R)$ with $H=N$. Then $G/H\simeq \R^2\bs\{0\}$ which we realize 
in $\mathbb{P}(V)$ where $V=\R^2\oplus\R$
via $v\mapsto [(v, 1)]$. The closure of $Z$ in $\mathbb{P}(V)$ consists of 
$Z$ and two closed orbits: The $G$-fixed point $[(0,1)]$ and the orbit
$\mathbb {P}(\R^2) \simeq G/P$ through $[(v_N,0)]$ where $v_N\in\R^2$ is $N$-fixed. 
\end{ex}

\par The goal of this section is to construct more suitable compactifications with a simple structure of the closed orbits.

\par Recall  the parabolic $Q=LU \supset P$ which we 
attached to $Z$.  In the sequel it is convenient to choose a Cartan involution $\theta$ on $G$ such that 
$L$ is $\theta$-stable. 
The opposite parabolic to $Q$ is then defined by $\oline Q:=\theta(Q) = L \oline U$ with $\oline{U}=\theta(U)$. 

\par If $V$ is a finite dimensional real $G$-module, then we denote by $V^*$ its dual.  We choose an inner product 
$\la\cdot,\cdot\ra$ on $V$ which is $\theta$-covariant, that is $\la g \cdot v, w \ra = \la v, \theta(g)^{-1}\cdot w\ra$ holds 
for all $g\in G$, and $v, w\in V$.  We consider the linear identification 
$$ V\to V^* , \ \ v\mapsto v^*:=\la \cdot, v\ra\, $$ 
and observe that the dual representation $V^*$ can be realized on $V$ but with the twisted $G$-action 
$g* v:=\theta(g)\cdot v$. In particular we see that the ray $\R^+ v\subset V $ 
is stabilized by $\oline Q$ if and only 
if $\R^+ v^*\subset V^*$ is stabilized by $Q$.

\par We call a compactification $\hat Z$ {\it simple} provided that there 
exists only one closed $G$-orbit $Y\subset \hat Z$ and $Y\simeq G/\oline Q$
as $G$-space. 

\par Simple compactifications arise in the following context.

\begin{lemma} \label{c1-lem} Let $V$ be an irreducible finite dimensional real rational $G$-representation with 
the following properties:
\begin{enumerate}
\item There is a non-zero vector $v_H\in V$ such that the stabilizer of the line $\R v_H$ is $H$.
\item There is a non-zero vector $v\in V$ for which the stabilizer of the ray $\R^+ v$ is $\oline Q$. 
\end{enumerate}
Then the closure $\oline{G\cdot [v_H]} \subset \mathbb{P}(V)$ is a simple compactification 
of $Z=G/H$ with closed orbit $Y=G\cdot[v]$. 
\end{lemma}

\begin{proof} The orbit $Y$ is isomorphic to $G/\oline Q$ by
  assumption. Since $ \oline Q$ is a parabolic subgroup, $Y$ is compact and
  therefore closed in $\mathbb{P}(V)$. The action of $G$ on
  $\mathbb{P}(V)$ is algebraic, hence all orbits are locally
  closed. In particular, $\oline{G\cdot [v_H]}$ must contain a closed
  orbit. Thus it suffices to show that $Y$ is the only closed
  $G$-orbit of $\mathbb{P}(V)$. This is a standard fact whose proof, for the 
convenience to the reader, we briefly recall: Let $G\cdot[u]$ be a closed orbit and  
decompose $u$ in weight spaces for $\af$. Since $V$ is irreducible we may assume
$u$ has a non-trivial component in the lowest weight space $[v]$. A sequence $a_n$ of
elements converging to infinity in the positive Weyl chamber $A^+$ will now exhibit $[v]$ as the 
limit of $a_n\cdot[u]$. 
\end{proof}

\subsection{The structure of $P_\C H_\C$}

Let $P$ be a minimal parabolic subgroup such that $PH$ is open in $G$. Then
$P_\C H_\C$ is open in $G_\C$ and we have:

\begin{lemma} Let $H_{\C,0}$ be the identity component of $H_\C$. Then
 $$P_\C H_\C=P_\C H_{\C,0}.$$
\end{lemma}

\begin{proof} This follows from the fact that $P_\C H_{\C,0}$ is Zariski open in $G_\C$ 
and the fact that $G_\C$ is irreducible. 
\end{proof}

We denote by
$\mathcal{P}_+^\C$ the multiplicative monoid of regular functions
on $G_\C$ which have no zero in $P_\C H_\C$. Every $f\in\mathcal{P}_+^\C$ is of the form
\begin{equation}\label{f product}
f(ph)=f(e)\chi(p) \psi(h)\qquad p\in P_\C, h\in H_{\C,0}
\end{equation}
with algebraic  characters
$\chi: P_\C \to \C^*$ and $\psi: H_{\C,0} \to \C^*$. This follows from 
Rosenlicht's theorem, see \cite{KKV} p.~78.  

\begin{lemma}\label{lemma P++} 
The Zariski closed subset $G_\C - P_\C H_\C$ in $G_\C$ is affine and the zero locus 
of a regular function on $G_\C$. 
\end{lemma}

\begin{proof} Recall that a complex homogeneous space $D_\C/C_\C$ of a reductive group $D_\C$ by a reductive 
subgroup $C_\C$ is affine. Hence the local structure theorem (\ref{LSTglob2})
implies that $P_\C\cdot z_0  \subset Z_\C$ is affine.  
It follows that the complement of $P_\C \cdot z_0$ is of pure codimension one in $Z_\C$ and likewise for 
$P_\C H_\C\subset G_\C$ (see \cite{Go}, Prop. 1).      

\par It remains to show that the divisor  ${\mathcal D}:=G_\C - P_\C H_\C$ is the zero locus 
of a regular function. This is a consequence of the fact that the Picard group of $G_\C$ 
is finite (cf. \cite{KKV}, Prop. 4.5). 
\end{proof}

\begin{definition}\label{defi P++}
We denote by $\mathcal{P}_{+}\subseteq\mathcal{P}_+^\C$ the set of regular functions $f$ on $G_\C$ 
for which 
\newcounter{saveenum}
\begin{enumerate}
\item\label{f(G) real} $f(G)\subseteq\R$, 
\item\label{M invariant} $f(mg)=f(g)$ for all $m\in M, g\in G$,
\item\label{P+ cond} $P_\C H_\C\subseteq \{g\in G_\C\mid f(g)\neq 0\}$.
 \setcounter{saveenum}{\value{enumi}}
\end{enumerate}
and by $\mathcal{P}_{++}\subseteq \mathcal{P}_{+}$ the subset with {\rm (\ref{P+ cond})} replaced by
\begin{enumerate}
 \setcounter{enumi}{\value{saveenum}}
\item\label{P++ cond} $P_\C H_\C= \{g\in G_\C\mid f(g)\neq 0\}$.
\end{enumerate}
\end{definition}

\begin{lemma}\label{lemma P++2}
 $\mathcal{P}_{++}$ is not empty.
\end{lemma}

\begin{proof}
We obtain from Lemma \ref{lemma P++}
a
function $f\in\C[G_\C]$ with $f(z)=0$ if and only if $z\notin P_\C H_\C$.
For elements in $P_\C H_\C=P_\C H_{\C,0}$ the identity (\ref{f product}) holds.
Let $g\mapsto \bar g$ denote the complex conjugation of $G_\C$
with respect to $G$. Then the function $F(z)=f(z)\oline{f(\bar z)}$
satisfies (\ref{f(G) real}) and (\ref{P++ cond}). 
As $M$ is compact, we observe that $|\chi(m)|^2 =1$ for all $m \in M$. Hence $F$ 
satisfies (\ref{M invariant}) as well.  
\end{proof}

\subsection{The left and right stabilizers of $P_\C H_\C$}

In this subsection our concern is with the 
left and right stabilizers of the double coset $P_\C H_\C$ in $G_\C$. 
We denote by $L$ and $R$ the left and right regular representations 
of $G_\C$ on $\C[G_\C]$, that is, 
$$(L(g)f)(h)=f(g^{-1}h),\qquad (R(g)f)(h) = f(hg),$$ 
for $g,h\in G_\C$ and $f \in \C[G_\C]$. 

We begin with the discussion of the left stabilizer.

\begin{lemma}\label{rmk Q}
Let $H<G$ be a spherical subgroup 
and let $P$ and $Q$ be as above. Then
\begin{equation}\label{eq QC}
Q_\C=\{ g\in G_\C\mid  gP_\C  H_\C  = P_\C  H_\C \}\,. 
\end{equation}
Furthermore, for $f\in \mathcal{P}_{++}$ one has: 
\begin{equation}\label{eq Q}
Q=\{ g\in G\mid L(g) f \in \R^+ f\}.
\end{equation}
\end{lemma}

\begin{proof} 
In \cite{KKS}, proof of Theorem 2.2, the parabolic subgroup $Q$ is obtained in an iterative
procedure, by which a strictly decreasing sequence of real parabolic subgroups 
$Q_0\supset Q_1\supset \dots \supset Q_k=Q$ is constructed. The parabolic subgroup
$Q$ obtained in the final step is unique, and in particular independent of the initial subgroup 
$Q_0$. Following the proof in \cite{KKS}
it can be seen that in each step of the iteration, the set
$(Q_{j+1})_\C H_\C$ is a proper subset of the previous set $(Q_j)_\C H_\C$. 
In the first step of the iteration, as described in \cite{KKS},
this amounts to the fact that the function
$F$ constructed there must satisfy $F(g)=0$ for some $g\in (G_n)_\C$, since it is a matrix coefficient for a 
non-trivial irreducible representation of this semisimple group.

We apply this procedure with $Q_0:= Q_{0, \C}\cap G$ where
$$Q_{0, \C}=\{ g\in G_\C \mid  gP_\C  H_\C  = P_\C  H_\C \}.$$
Note that $Q_{0, \C}$ 
is a parabolic subgroup of $G_\C$ which is defined over $\R$, and hence
$Q_0$ is a real parabolic subgroup with $(Q_0)_\C=Q_{0,\C}$. 
It follows from $QH=PH$ that $Q\subseteq Q_0$. If $Q$ is strictly smaller than $Q_0$ 
the discussion above implies that $Q_\C H_\C$ is strictly smaller than
$Q_{0,\C} H_\C$, which contradicts the definition of $Q_{0,\C}$. Hence  
$Q=Q_0$ and (\ref{eq QC}) is valid.

Finally let $f\in\P_{++}$ and note that
$$Q\subseteq \{ g\in G\mid L(g) f \in \R^+ f\}\subseteq Q_0$$
by (\ref{f product}) and Definition \ref{defi P++}. 
Hence (\ref{eq Q}) also follows.
\end{proof}

\begin{rmk} The corresponding real version of (\ref{eq QC}),
in which $P_\C H_\C$ is replaced by $PH$, is not true in general.
For instance if $H=K$ is a maximal compact subgroup of $G$, then $PH=G$ and the left stabilizer 
in $G$ of $PH$ is $G$ whereas $P=Q$ in this case. 
\end{rmk}

We move on to the discussion of the right stabilizer of $P_\C H_\C$
which we write as  $J:=J_\C \cap G$ where
\begin{equation}\label{defJ}
J_\C:=\{ g\in G_\C \mid P_\C H_\C g = P_\C H_\C\}.
\end{equation}
Furthermore, for $f\in {\mathcal P}_{++}$ we set 
\begin{equation}\label{Hfx}
\Hfx:=\{ g\in G \mid R(g) f \in \R^\times f\}\, .
\end{equation}

Note that both $J$ and $\Hfx$ are closed algebraic 
subgroups of $G$

\begin{lemma}\label{J lemma} One has  
\begin{equation}\label{J inclusion}
H_{\C,0} \cap G \,\subset \Hfx \subset J
\end{equation}
and 
\begin{equation}\label{normalizer inclusion}
H \subset N_G(H)\subset J\, .
\end{equation}
In particular, $\Hfx$, $N_G(H)$, and $J$ are spherical subgroups.
\end{lemma}

\begin{proof}
The inclusions in (\ref{J inclusion}) follow from
(\ref{f product}) and from Condition (\ref{P++ cond}) of Definition \ref{defi P++}, respectively.

  The normalizer of $H$ acts on the set of open $H_\C$-orbits in
  $P_\C\backslash G_\C$. Since there is only one, (\ref{normalizer inclusion}) follows.
\end{proof}

\begin{ex} It can happen that $J/N_G(H)$ is of positive dimension. 
For that let $G=\SO(1,n)$ with  $n\geq 3$ and let  $P=MAN$ be a minimal parabolic subgroup of $G$.
Then $M= \SO(n-1)$. Let $M' <M$ be any subgroup which is transitive on 
the $n-2$-sphere.  For instance if $n-1=2k$, then $M'= U(k)$ is such a group. 
Further, $H= M' A$ is a self-normalizing spherical subgroup with $J= MA$.
(see \cite{KS1} for all that).
\end{ex}

\par For a topological group $G$ we denote by $G_0$ its identity component.

\begin{lemma}\label{H_f=J} The following assertions hold: 
\begin{enumerate} 
\renewcommand{\labelenumi}{(\alph{enumi})}
\item For $f\in {\mathcal P}_{++}$ one has 
$$(\Hfx)_0 = J_0\, .$$ 
In particular $\hfx=\jf$ is independent of the choice of 
$f\in {\mathcal P}_{++}$. 
\item There exists 
$f \in {\mathcal P}_{++}$ such that 
$$ \Hfx = J\, .$$
\end{enumerate}

\end{lemma}

\begin{proof} (a) The inclusion $\subseteq$ is obvious from (\ref{J inclusion}).
The function $f$ is non-zero, hence invertible on $P_\C
H_\C\supset J_\C$. Any invertible function is an eigenfunction for a connected
group (see e.g \cite{KKV} Prop. 1.3). Hence $J_0\subseteq \Hfx$. 
\par As for (b) note that $J_0 =(\Hfx)_0$. Let $g_1,\ldots, g_k\in J$ be elements such that 
$J= \bigcup J_0 g_j$, with a disjoint union. 
Let $ f \in {\mathcal P}_{++}$. Then $F:=\prod_{j=1}^k R(g_j) f $ lies in ${\mathcal P}_{++}$
and has the desired property.  
\end{proof}

\begin{lemma}\label{lem hf} We have $$\hf\subset \jf  \subset \hf +\df.$$
\end{lemma}

\begin{proof} 
By Lemma \ref{H_f=J} it suffices to show for $f\in {\mathcal P}_{++}$ that
\begin{equation}  \hfx  \subset \hf +\df\, .\end{equation}
Let $\phi:Z=G/H\to Z_f:=G/\Hfx$ be the canonical
  map. Let $Q_f\supseteq P$
  the unique $Z_f$-adapted parabolic above $P$ (see Section \ref{lst}).
  It follows from the
  construction of the adapted parabolics (which we just summarized in
  the proof of Lemma \ref{rmk Q}) that $Q=Q_f$. Now we apply the
  local structure theorem to $Z$ and $Z_f$. Let $S\subseteq Z$ and
  $S_f\subseteq Z_f$ be slices. Since they are constructed from the
  same $f$, we have $S=\phi^{-1}(S_f)$. Since $S$ is homogeneous for
  $D$, this implies the assertion $\df+\hf=\df+\hfx$.
\end{proof}

Recall the notion of real rank of $Z$ from \cite{KKS} which is an invariant 
of the real spherical space $Z$ and given by $\rank_\R Z = \dim \af_Z$. 

\par Note that if $Z$ is quasi-affine, then the complement of the
open $P_\C$-orbit in $Z_\C$ is the zero locus of some
$P$-semiinvariant regular function on $Z$. Its pull-back $f$ to $G$ is
therefore in $\mathcal{P}_{++}$ with trivial right character $\psi$.
We denote by ${\mathcal P}_{++,
  \1}$ the subset of ${\mathcal P}_{++}$ which corresponds to
$\psi=\1$ on $H_{\C,0}$, and set for this case
$$I:= \bigcap_{f \in \mathcal P_{++, \1}} H_f , \qquad 
H_f:=\{ g\in G\mid R(g) f = f \}.$$

Then $H_{\C,0}\cap G\subset I$ and 
\begin{equation}\label{HIJ}
I\subset H_f\subset J.
\end{equation} 
We obtain a refinement of Lemma \ref{lem hf} as follows. 

\begin{prop} \label{prop hf} Suppose that $Z$ is quasi-affine. 
Then the following assertions hold: 
\begin{enumerate} 
\item\label{first}  $\rank_\R G/H = \rank_\R G/I$. 
\item\label{second} $\ifr\subseteq \hf +\df_c$. 
\item\label{third}  $\jf\subseteq \ifr  +\af_Z$. 
\end{enumerate}
\end{prop}

\begin{proof} By comparing $P$-semiinvariants of $G/H$ on the slice,
  the local structure theorem implies that $\rank_\R G/H$ is also the
  rank of the lattice spanned by $P_{++, \1}$. Since by definition
  these are the semiinvariants of $G/I$, this shows the first
  assertion.  Moreover, the equality of ranks shows that the map
  between the slices of $G/H_0$ and $G/I_0$ has compact fibers. This implies 
 the second assertion. The last assertion follows from the
  fact that $\hfx \subseteq \hf_f +\af_Z$ for all $f\in \P_{++,\1}$.
\end{proof}

If $G$ is a real algebraic group, then we denote by $G_{\rm n}$ the normal 
subgroup generated by all unipotent elements. 
Note that $G_{\rm n}$ is connected, and that if $G=L\ltimes R_u$ is an arbitrary Levi 
decomposition, then $G_n=L_n\ltimes R_u$, where $L_n$ is as defined in Section \ref{lst}.
In particular, $G/G_{\rm n}$ has compact Lie algebra. 

In the sequel we view the elements of $\C[Z_\C]$ as right $H_\C$-invariant regular functions 
on $G_\C$. 

\begin{lemma} \label{J-inv}The space $\C[Z_\C]=\C[G_\C]^{H_{\C}} $ of right $H_\C$-invariant regular 
functions on $G_\C$ is right $J_{{\rm n}, \C}$-invariant.  
\end{lemma}

\begin{proof} Let $B_\C \subset P_\C$ be 
a Borel subgroup. To prove the lemma it is sufficient to show that every left 
$B_\C$-eigenspace in $\C[Z_\C]$ 
is fixed by $J_{{\rm n}, \C}$ under the right regular action.  Let $f$ be such an eigenfunction. 
Suppose first that $f$ is in fact 
a $P_\C$-eigenfunction, i.e.,~$f\in \mathcal{P}_+^\C$ is attached to a pair of characters $(\chi, \psi)$
on $P_\C H_\C=P_\C H_{\C,0}$. 
As in the proof of Lemma \ref{H_f=J}(a)
we obtain that $R(n)f$ is a multiple of $f$ for all $n \in (J_\C)_0$.

\par In general a $B_\C$-eigenfunction does not need to be a $P_\C$-eigenfunction.  To overcome this difficulty we 
use the method of $M$-averages of \cite{KKS}: 

Let $\C[G]$ denote the ring of regular functions on $G$, i.e.,~the restrictions
to $G$ of functions in $\C[G_\C]$.
Using the $M$-average
$$F\mapsto F^M(g) = \int_M F(mg) \ dm, \quad (g \in G)$$
(with normalized Haar measure $dm$ on the compact group $M$),
we obtain a quadratic map 
$$ \C[G_\C] \to \C[G]^M, \ \ \phi\mapsto \tilde \phi; \quad \tilde \phi := (\big|\phi|_G \big|^2)^M\, .$$

If $f$ is a left $B_\C$-eigenfunction, then $\tilde f$ is a left $P_\C$-eigenfunction 
and by what we just established we conclude that 
$$ R(n)\tilde f = \tilde f \qquad (n \in J_{{\rm n},\C})\, .$$
Further observe that by right equivariance of $\phi\mapsto\tilde\phi$ 
\begin{equation} \label{id} [R(n) f]^\sim = R(n) \tilde f = \tilde f\qquad (n\in J_{{\rm n}, \C}) \, .\end{equation} 
 
\par Let $u(t), t\in \R$, be 
a unipotent one parameter subgroup of $J_{\rm n}$.  Note that $V:=\Span\{R(u(t)) f\mid t\in \R\}$ is a finite 
dimensional subspace. We find elements $F_0, \ldots, F_N$ of $V$ with $F_0=f$ and $F_N\neq 0$ such that 
$$ R(u(t))f= F_0 + t F_1+ \ldots + t^N F_N,\quad (t\in\R).$$ 
When applying the quadratic map from above we obtain
$$[R(u(t))f]^\sim = t^{2N}\widetilde{F_{N}} + \hbox{lower order terms in $t$}\, .$$
From $F_N\neq 0$ it follows that $\widetilde{F_N}\neq 0$, and from
(\ref{id}) we then deduce $N=0$.
The lemma is proved.
\end{proof}

\begin{cor}\label{J-cor}  $H_{\rm n}=J_{\rm n}$. 
\end{cor}

\begin{proof} We may assume that $Z$ is quasiaffine. 
First it is clear that $H_{\rm n}\subseteq J_{\rm n}$. To obtain the converse, 
let us denote by $\Aut (Z_\C)$ the group of birational automorphisms of $Z_\C$ and by 
$\Aut_{G_\C}(Z_\C)$ the subgroup of $G_\C$-equivariant ones therein.  Note that 
$N_{G_\C}(H_\C)$ naturally identifies with $\Aut_{G_\C}(Z_\C)$. 
The group $\Aut(Z_\C)$ is isomorphic to the automorphism group of the 
function field $\C(Z_\C)$. The fact that $Z$ and hence $Z_\C$ are quasiaffine implies that 
$\C(Z_\C)$ is the quotient field of $\C[Z_\C]$. 
According to Lemma \ref{J-inv} the group $J_{{\rm n}, \C}$ preserves $\C[Z_\C]$ and hence 
$J_{{\rm n}, \C}\subseteq N_{G_\C}(H_\C)$. As $N_{G_\C}(H_\C)/ H_\C$ has no real unipotent elements (\cite{KKS} Cor. 4.3),
the corollary follows. 
\end{proof}

\par Recall the finite set $F$ from (\ref{finite set}).  
The following result, which is analogous to
\cite{BP}, Prop. 5.1, implies the first assertion of Theorem \ref{mainthm2}.

\begin{theorem} \label{J-cpct} Assume $N_G(H)=H$. Then there exists a compact 
subgroup $M_J<J$ with Lie algebra ${\mathfrak m}_J \subseteq \df_c$ and 
\begin{equation}\label{J/H}
J= M_J H. \end{equation}
Furthermore one has 
$$ J \subseteq  (M\cap J) F H \, .$$ 

\end{theorem}

\begin{proof} 
The group $J/J_n$ is compact modulo its center, and thus 
$J=M_J Z_1$ where $M_J$ is compact and where
$Z_1=\{g\in J\mid gJ_n\in Z(G/J_n)\}$.
Further, Proposition \ref{prop hf}(\ref{second})-(\ref{third}) 
implies that ${\mathfrak m}_J \subseteq \df_c$. 

Note that if $gJ_n\in Z(J/J_n)$ then $g\in N_G(H)$
since  $J_n\subset H\subset J$ by Corollary 
\ref{J-cor}. Since $H=N_G(H)$ we conclude that 
(\ref{J/H}) is valid.

\par 
We have $J_\C\subset Q_\C H_\C $ and hence 
$$J_\C\cdot z_0  = (Q_\C \cap J_\C )\cdot z_0=   (L_\C  \cap J_\C )\cdot z_0.$$
As $L_{n, \C} \subset H_\C$ and $M_\C A_\C \to L_\C/ L_{n, \C}$ is onto we conclude that 
$$J_\C \cdot z_0=(M_\C A_\C\cap J_\C)\cdot z_0.$$
The last assertion of the theorem now follows by taking real points. 
\end{proof}

\subsection{The existence of simple compactifications} \label{S:exist simple c}
We are now ready to prove the second assertion of Theorem \ref{mainthm2}. By Lemma
\ref{H_f=J} there exists 
$f \in {\mathcal P}_{++}$ such that 
$ \Hfx = J.$

\begin{theorem}\label{exists simple compactification} 
Let $Z=G/H$ be a real spherical space. Let $f\in \mathcal{P}_{++}$ such that $\Hfx=J$ and let
$V_f:=\Span_\R\{ R(g)f\mid g\in G\}$ be the real representation of 
$G$ generated by $f$. Then 
\begin{equation}\label{Vf embedding}
 Z_J:=G/J \to \mathbb{P}(V_f), \ \ gH\mapsto R(g)f
\end{equation}
is a $G$-equivariant embedding. 
The closure $\hat Z_f$
of the image provides a simple compactification of $Z_J$.  
\end{theorem}

\begin{proof} As $J=\Hfx$, the map (\ref{Vf embedding})
determines a $G$-equivariant embedding 
from $Z_J$ into $\mathbb{P}(V_f)$. 
Note that $V_f$ is irreducible and that
$f$ can be written as a matrix coefficient, that is 
$f(g)= v_0^*(g\cdot v_H)$ for some $J$-eigenvector $v_H \in V= V_f$ and a $Q$-eigenvector 
$v_0^*$ in the dual $V^*$. More precisely, $v_H$ is $f$ and $v_0^*$ is the evaluation at
$\1$. It follows from Lemma \ref{rmk Q} that $Q$ is the stabilizer of 
$\R^+v_0^*$.
With Lemma \ref{c1-lem} we thus obtain the result.
\end{proof}

\subsection{Coordinates near the closed orbit}\label{SCII}

In this section we assume that $H=J$.
Further we let $\hat Z=\hat Z_f$ be a simple  compactification constructed out of a function 
$f\in \mathcal{P}_{++}$ as in Theorem \ref{exists simple compactification}. 
Let $Y$ be the unique closed orbit in $\hat Z$. 
We recall 
that if $Q=LU$ is a Levi decomposition of the type
discussed in Sections \ref{lst}--\ref{open orbits}, we have defined
$D:=L/L_n$ with compact Lie algebra $\df=\zf(\lf) + \lf_c$.

\begin{prop} \label{localcoo} There is an affine $Q$-invariant
open subset $\hat Z_0\subset\hat Z$, 
a Levi decomposition $Q=LU$ as mentioned above,
and an affine $L$-invariant subvariety 
$S_Y$ of $\hat Z_0$ such that: 
\begin{enumerate} 
\item\label{eins} The decomposition {\rm (\ref{LSTglob})} is valid with $S=L\cdot z_0\subset S_Y$. 
\item\label{zwei} $ U \times S_Y \to  \hat Z_0, \ \ (u,s)\mapsto u\cdot s$ is a homeomorphism. 
\item\label{drei} $S_Y$ is pointwise fixed by $L_{\rm n}$ and decomposes into finitely 
many $D$-orbits. 
\item\label{vier} $\hat Z_0\cap Z$ is the union of the open $P$-orbits in $Z$,
and $S_Y\cap Z$ consists of the open $D$-orbits in $S_Y$.
\item\label{funf} $S_Y\cap Y = \{\hat z\}$
 where $\hat z$ is the origin of $Y=G/\oline Q$.
\end{enumerate}
\end{prop} 

\begin{proof}  In \cite{KKS}, Section 3 and in particular Th. 3.11,  we derived a local 
structure theorem for real algebraic varieties which applies to $\hat Z$.
The construction was explicit and based on finite dimensional representation theory.
We use the finite dimensional representation $V:=V_f$ with $H$-semi-spherical vector 
$f:=v_H$. Further we pick the $Q$-eigenvector 
$v_0^*$ in the dual of $V^*$ of $V$. With 
the data $(V, v_H, v_0^*)$ we begin the construction of the local slices as in 
\cite{KKS}, Section 3.  
One obtains an affine open set $\hat Z_0\subset \hat Z$ by 
$$ \hat Z_0=\hat Z\cap\{[v]\in \mathbb{P}(V)\mid v_0^*(v)\neq 0\}$$ 
and a  $Q$-equivariant moment type map 
\begin{equation}\label{momap} \mu: \hat Z_0 \to \gf^*;\  \mu([v])(X):= {v_0^*(X\cdot v)\over v_0^*(v)}\, .
\end{equation}
 A Levi decomposition $Q=LU$ that leads to (\ref{LSTglob})
is then obtained with $L$ being the stabilizer in of $\mu(z_0)=\mu([v_H])\in \gf^*$
in the coadjoint representation of $Q$, and with $S:=\mu^{-1}\{\mu(z_0)\}\cap P\cdot z_0$. 
It follows from (\ref{LSTglob}) that $S=L\cdot z_0$, and
(\ref{eins}) is established except for the final inclusion.

With these choices we let $\oline Q$ be the parabolic subgroup opposite to $Q$.
Let $\hat z=[v_0]$ for the (up to scalar) unique $\oline Q$-eigenvector 
$v_0\in V$. 
Note that $v_0^*(v_0)\neq 0$ and thus $\hat z\in \hat Z_0$.  
The slice $S_Y$ is defined by $\mu^{-1} \{ \mu (\hat z)\}$. 
In particular, since $\hat z$ is fixed by $L$,
the stabilizer in $Q$ of $\mu(\hat z)\in\gf^*$
is a Levi subgroup that contains and hence equals $L$.
Note that we do not need any iterations of the construction as in 
\cite{KKS} as $Q$ is already the $Z$-adapted parabolic.  The 
assertions in  (\ref{zwei}) and (\ref{drei}) now follow from \cite{KKS}, Th. 3.10 and 3.11.

 It follows from (\ref{zwei}) that there exists unique 
elements $u\in U$ and $s\in S_Y$ such that $z_0=u\cdot s$. Then
$\mu(z_0)=\Ad^*(u)\mu(s)= \Ad^*(u)\mu(\hat z)$ and hence
$\mu(z_0)$ is stabilized by $uLu^{-1}$. Hence $L=uLu^{-1}$ and
since $U$ acts freely on the set of Levi subgroups in $Q$ we conclude
$u=\1$. Then $z_0=s\in S_Y$ and the final inclusion of (\ref{eins}) follows.

\par The first assertion in (\ref{vier}) follows from the fact that $f\in \mathcal{P}_{++}$. Indeed 
the construction in \cite{KKS} yields that $\hat Z_0 \cap Z$ coincides
with the non-vanishing locus of $f$, i.e.,~the union of all open 
$P$-orbits in $Z$ (see Lemma \ref{open orbits lemma}). Because of  (\ref{zwei})-(\ref{drei}), 
the second assertion of (\ref{vier}) is an immediate consequence of the first.

\par For  (\ref{funf}) we remark first that 
$\hat z$ belongs to $S_Y$ and 
is $\oline Q$-fixed, in particular $D$-fixed. It remains to show there are no other
elements from $Y$ in $S_Y$.
Let $\hat y = [g\cdot v_0]\in Y$ for some $g\in G$. As $G=U \mathcal{W}_A \oline Q$ for the 
Weyl group $\mathcal{W}_A$ attached to $A$, we  
we may assume that $g=u w$ for some $w\in \mathcal{W}_A$ and $u\in U$.  
If also $\hat y\in S_Y$, then $v_0^*(g\cdot v_0)\neq 0$. Hence we may assume
$w=\1$ as $v_0^* (uw\cdot v_0)=v_0^*(w\cdot v_0)$, which is zero if $w\cdot v_0\neq v_0$. 
We claim that $u=\1$. 

Assume $u\neq \1$ and set $u=\exp(Y)$ for $Y\in\uf$. 
Now $\hat y\in S_Y$ implies $\mu(\hat y) = \mu(\hat z)$ and so 
\begin{equation} \label{L-R} 
v_0^*(X\cdot v_0)=  v_0^*(X\cdot u\cdot  v_0)
\end{equation}
for all $X\in \gf$ by (\ref{momap}). For $X\in \uf$ 
the left hand side of (\ref{L-R}) vanishes, whereas the 
other side is
$$v_0^*(\Ad(u^{-1})X\cdot  v_0)=v_0^*( e^{-\ad Y} X\cdot v_0)\,.$$
Let $p: \gf \to \zf(\lf)$ be the projection along $\uf +\oline{\uf} + [\lf, \lf]$ and observe that 
$$v_0^*(e^{-\ad Y} X \cdot v_0) = \lambda(p(e^{-\ad Y} X)) v_0^*(v_0)$$
where $\lambda\in \af^*$ is the $\af$-weight of $v_0$. Hence
$\lambda(p(e^{-\ad Y} X))=0$.
Write $Y=\sum Y_\alpha$ as a sum of $\af$-weight vectors and let $\beta$ be the smallest root for which 
$Y_\beta\neq 0$. Then for $X=\theta(Y_\beta)$ we obtain
$$p(e^{-\ad Y} X)= - [Y_\beta, \theta(Y_\beta)].$$ 
As $\bar Q$ is the stabilizer of $[v_0]$, the weight $\lambda$ is non-zero
on $[Y_\beta, \theta(Y_\beta)]$ for every root $\beta$ of $\uf$.
We reached a contradiction, hence $u=\1$ and (\ref{funf}) holds.
\end{proof}

Later in the text we will use Proposition \ref{localcoo} in a slightly disguised form 
which we explicate in the following corollary. 

As in {\rm (\ref{allorbits})} let  $t_1, \ldots, t_m \in T_Z$ parametrize 
the open $P$-orbits on $Z$ and  let $S_{Z,j}$ be the closure of $S_{Z,j}'= D t_j \cdot z_0$ 
in $\hat Z_0$. We 
set 
$$S_Z:=\bigcup_{j=1}^m S_{Z,j} \quad \hbox{and} \quad S_Z':= \bigcup_{j=1}^m S_{Z,j}'\, .$$

\begin{cor}
\label{corloc} Within the notation of Proposition \ref{localcoo}, 
 $S_Z=S_Y$. In particular the following assertions
hold: 
\begin{enumerate}
\item $\hat Z_0 = U S_Z$. 
\item $S_Z \cap Z = S_Z'$. 
\item $S_Z \cap Y = \{ \hat z\}$.
\end{enumerate}
\end{cor}
 
\begin{proof}
Since $z_0\in S_Y$ we find $S_{Z,1}\subset S_Y$. 
By Proposition \ref{localcoo} there exist for each $j$ unique 
elements $u_j\in U$ and $s_j\in S_Y$ such that $t_j\cdot z_0=u_j\cdot s_j$. 
Then $s_j\in S'_{Y,j}:= S_Y\cap (Pt_j\cdot z_0)$ and hence the
$D$-orbits $u_j\cdot S_{Y,j}'$ and $S_{Z,j}'$ agree. 
By (\ref{LSTglob2}) the elements $u_j\in U$ all have to be equal 
to $u_1=1$. Hence $S_Z= S_Y$. 
\end{proof}

\section{Polar decomposition}\label{pode}
We  recall that for a minimal 
parabolic subgroup $P$ with $PH$ open, we chose a 
specific Cartan involution $\theta$ of $G$, which 
was adapted to the geometry of $PH$.
The corresponding Langlands decomposition of $P$
is denoted $P=MAN$, and the corresponding maximal compact subgroup of $G$ 
is denoted by $K$. Then $M\subset K$. Recall the finite set $F$ from (\ref{finite set})
and its property (\ref{fs2}). 

\begin{lemma} \label{lemma1} Suppose that $Z=G/H$ is a homogeneous  real spherical space and 
$P$ a minimal parabolic subgroup of $G$ such that $PH$ is open.
Then there exist finite sets $F', F''\subset G$  such that 
 $$G= F' K A_Z  F'' H\, .$$
Moreover $PgH$ is open for all $g\in F''$, and
in case $H=N_G(H)$ one has $F''=F$. 
\end{lemma}

\begin{proof}  Recall from \cite{KKS} Prop.~4.2
that $N_G(\hf)_0\subset MA_ZH$. Hence it suffices to show
 $$G= F' K A_Z  F N_G(\hf)_0$$
and we may assume that $\hf$ is self-normalizing.  
It then follows that
$N_G(N_G(\hf))=N_G(\hf)$ and hence we can apply 
 Theorem \ref{J-cpct} to the spherical subgroup
 $N_G(\hf)$. 
 In this context observe also that $PgH$ is open for all $g \in MA F N_G(\hf)$. 
  By that we can then reduce all statements to the 
case where $H=J$, which we assume henceforth.

\par We let $\hat Z=\hat Z_f$ be a simple compactification 
as constructed before. 

\par Let $Y\subset \hat Z$ be the the unique closed boundary orbit. 
We use Proposition \ref{localcoo} and its Corollary \ref{corloc} to obtain a 
$P$-stable affine open set 
$\hat Z_0=US_Z\subset \hat Z$ 
with the properties (1)-(3) listed in Corollary \ref{corloc}. 

Consider the map 
$$\Phi: K \times S_Z \to \hat Z, \ \ (k, s) \mapsto k\cdot s.$$

\begin{lemma} \label{lemma2}
The map $\Phi$ is open at $\hat z$.
\end{lemma}

\begin{proof}
This is easily seen if $\hat z$ is a smooth point of the affine variety $S_Z$. 
Indeed, as the projection $\kf\to\uf\simeq \gf/\bar\qf$ is surjective, 
we see that 
the  differential
$d\Phi$ is surjective at 
$({\bf 1},  \hat z)$. 

We now consider the general case. As in \cite{BLV}, Sect. 1.2, 
the slice is induced from a linear slice in the surrounding projective space $\mathbb{P}(V_f)$ of 
$\hat Z$: Let 
$\mathbb{P}(V_f)_0:=\{ [v]\in \mathbb{P}(V_f)\mid v_0^*(v)\neq 0\}$ and note that 
$\mathbb{P}(V_f)_0\simeq v_0 + \ker v_0^*$ is  affine open. In analogy 
to Prop. 1.2 in \cite{BLV} we have that 
$$ \Psi: U \times  \big( \R^\times v_0 + (\gf \cdot v_0^*)^\perp\big) \to \mathbb{P}(V_f)_0,\ \  (n, w)\mapsto  [n\cdot w]$$
is a diffeomorphism. The slice $S_Z$ is obtained from the intersection of $\hat Z$ with 
$\R^\times \big(v_0 + (\gf \cdot v_0^*)^\perp\big)$. 
As $K\cdot \hat z= G/\oline Q$, it 
follows that the differentiable map 
$$\Psi': K \times ( \R^\times v_0 + (\gf \cdot v_0^*)^\perp\big) \to \mathbb{P}(V_f)_0, 
\ \ (k, w)\mapsto [k\cdot w]$$
has surjective differential at $(\1, v_0) \leftrightarrow \hat z$. 
In particular we can conclude that $\Phi$ is an open map near  $\hat z$. 
\end{proof}

We can now complete the proof of Lemma \ref{lemma1}.
As  $\hat z$ is $D_c$-fixed we find a $D_c$-invariant open neighborhood 
$S_{Z}^1 \subset S_Z$ containing  $\hat z$ such that  
$K\cdot S_{Z}^1=: {\mathcal U}$ contains an open neighborhood of $Y$ in $\hat Z$.
 
We may assume that $D_c\subset K$ and thus replace $S_Z$ by $S_Z/D_c$ which is a finite union of 
$A_Z$-orbits 
(the real points of a toric variety). These finitely many $A_Z$-orbits we realize in $S_Z$
as the set $S_{Z}^2$, i.e., we choose a continuous  embedding $S_Z/D_c \hookrightarrow S_Z$ and let 
$S_{Z}^2$ be the image. 
The piece in $S_{Z}^2$ which corresponds to $S_{Z}^1$ 
we denote by $S_{Z}^3$. Let ${S_{Z}^3}':= S_{Z}^3\cap Z$ and recall from Corollary \ref{corloc}
that 
\begin{equation} \label{intersec} {S_{Z}^3}'\subset DF\cdot z_0 = D_c A_Z F \cdot z_0\, . \end{equation}

\par  We find for every element $z\in \hat Z$ an 
element $g_z\in G$ such that $g_z^{-1} \cdot z \in S_{Z,3}$.  As $\hat Z$ is compact, we conclude that finitely 
many $g_z$, say $g_1, \ldots, g_N$  suffice so that $\bigcup_{j=1}^N g_j KS_{Z,3}=\hat Z$. 
We obtain
with (\ref{intersec}) the 
assertion of the lemma.
\end{proof}

We reformulate Lemma \ref{lemma1} in a more compact form and obtain:

\begin{theorem}  {\rm (Polar decomposition)} Suppose that $Z=G/H$ is a homogeneous  real spherical space and 
$P$ be a minimal parabolic subgroup of $G$ such that $PH$ is open. 
Then there exists a compact subset $\Omega\subset G$ and a finite set $F''\subset G$ such that 
$$G= \Omega A_Z  F'' H\, .$$
Moreover $PgH$ is open for all $g\in F''$. 
\end{theorem}

In particular, Theorem \ref{mainthm} follows.

\section{Compression cones and the fine convex geometry near the closed orbit}\label{section cone}

In this section we will investigate the fine convex geometry of the slice $S_Z$ 
near the $\oline Q$-fixed point 
$\hat z$.  

\subsection{Compression cones of $H$-spherical representations}

We begin with a remark about finite dimensional representations.  
The Cartan-Helgason theorem asserts that an irreducible finite dimensional representation 
of $G$ has an $MN$-invariant highest weight vector if and only if it is 
$K$-spherical. All finite dimensional irreducible representations considered in this section are assumed 
to be of this type. It is known that these representations admit a real structure  
and are self-dual. 

\par Let $(\pi, V)$ be finite dimensional real irreducible representation which is $H$-semi-spherical,
i.e.~there exists a non-zero vector $v_H \in V$  and an algebraic character $\chi: H\to\R^\times$
such that $H$ acts on $v_H$ by $\chi$. 
By the assumption above there exists a highest weight 
$\lambda\in \af^*$  of the dual representation $(\pi^*, V^*)$ 
and a highest weight vector
$v_\lambda^* \in V^*$ (unique up to scaling) with 
$$(man)\cdot v_\lambda^* =a^\lambda \cdot v_\lambda^*, \quad man\in P=MAN\, .$$
Here $a^\lambda= e^{\lambda(\log a)}$ as usual. In this case 
$\pi$ is said to be $(P, H)$-semi-spherical. Let
$$ f_{\pi}(g) := v_\lambda^* (\pi(g) v_H)$$
be the corresponding matrix coefficient, then 
$f_\pi$ belongs to the space
${\mathcal P}_+$ from Definition \ref{defi P++}, and every element in $\Pc_+$ is of this form
for some $\pi$ and $v_H$. 
In particular, the fact that $Z$ is real spherical implies that
$v_\lambda^* (v_H)\neq 0$, and this in turn implies that for a given character $\chi$ 
the semi-spherical vector $v_H$ is unique up to scaling.

Assume to begin with that $Z=G/H$ is a quasi-affine real spherical space. 
Let ${\mathcal P}_{+, \1}\subset {\mathcal P}_+$ be the submonoid corresponding to right $H$-invariant 
functions, it then consists of the functions $f_\pi$ as above with trivial character $\chi$.
Thus
\begin{equation}\label{P1+}
 {\mathcal P}_{+,\1}= \R[G/H]^{MN}
\end{equation}

\par Set $A_H:= A\cap H$. In the sequel we prefer to consider $A_Z$ as a quotient $A_Z = A/ A_H$.
Let $(\pi,V)$ be an $H$-spherical irreducible representation as above. Then
we define a cone $\af_{Z,\pi}^{--}\subset \af_Z$ by
\begin{equation}\label{open comp cone}
\af_{Z,\pi}^{--}=\{ X \in \af_Z  \mid  \lim_{t\to \infty}  [\pi(\exp(tX)) v_H] = [v_{-\lambda}]\in 
{\mathbb P}(V)\,\}\,.
\end{equation}
where $v_{-\lambda}\in V$ is a lowest weight vector.

\begin{definition}[Compression cone]\label{defi comp cone}
The closure $\af_{Z,\pi}^-$ of $\af_{Z,\pi}^{--}$ is
called the compression cone of the $H$-spherical representation $\pi$. 
Taking the intersection over all $H$-spherical representations $\pi$ as above, the set
\begin{equation}\label{comp cone}
\af_Z^- = \bigcap_{\pi } \af_{Z,\pi}^-\subseteq\af_Z
\end{equation}
is called the compression cone of $Z$.
\end{definition}

It will be seen later (below Lemma \ref{fin gen}) that the intersection is finite.
First we want to show that $\af_{Z,\pi}^-$ is a polyhedral convex cone.
Let us 
expand the $H$-spherical vector $v_H$ into $\af$-weights, say 
$$ v_H = \sum _{\mu} v_\mu$$
with $\pi(a)v_\mu = a^\mu\cdot  v_\mu$ for $a\in A$.  
In particular, $v_{-\lambda}\neq 0$ since $v_\lambda^* (v_H)\neq 0$
and $v_\lambda^*(v_\mu)=0$ for $\mu\neq-\lambda$. We may assume
(\ref{open comp cone}) refers to this vector.

\begin{rmk} \label{edge remark} Note that $a^\mu =1$ 
for all $a\in A_H$ with $v_\mu\neq 0$.\end{rmk}

\par Write $\Sigma_\uf\subset \Sigma^+$ for the roots 
corresponding to $\uf$. As $ V= {\mathcal U}(\uf) v_{-\lambda}$ it follows that  all $\af$-weights of $V$ are 
contained in $ -\lambda + \N_0[\Sigma_\uf]$.
Hence we find a finite subset 
$\Lambda_{\pi} \subset \N_0[\Sigma_\uf] $ such that 
$$v_H = \sum_{\nu \in \Lambda_\pi}  v_{-\lambda+\nu}$$
and $v_{-\lambda+\nu}\neq 0$ for $ \nu \in \Lambda_{\pi}$.
Since $v_{-\lambda}\neq 0$ we obtain:

\begin{lemma} \label{dualcone}Let $(\pi, V)$ be an irreducible real $H$-spherical representation.
Then for $X\in \af_Z$ the following statements are equivalent:

\begin{enumerate} 
\item $X\in \af_{Z, \pi}^{--}$
\item $ \nu (X) <0$ for all $\nu \in \Lambda_\pi\setminus\{0\}$.
\end{enumerate}
In particular, $\af_{Z,\pi}^-$ is a finitely generated closed cone.
\end{lemma}

\begin{rmk} Let  $\af^{--}$ be  the negative Weyl-chamber with respect to $\Sigma^+$, i.e.
$$ \af^{--}:= \{ X\in \af \mid (\forall \alpha\in \Sigma^+) \ \alpha(X)<0\}\, .$$
Then by the preceding lemma
\begin{equation}\label{WC} (\af^{--} +\af_H )/ \af_H \subseteq \af_{Z,\pi}^{--}\, .  \end{equation}
There are various instances when the inclusion (\ref{WC}) is strict. For example,  if 
$\pi$ is the trivial representation. A more serious obstruction will be encountered in the 
next section when we discuss wave front spherical spaces.
\end{rmk}

Let us now describe how the $\af_{Z,\pi}^-$ behave under tensor products. 
For notational reasons we prefer to write $\af_{Z,\lambda}^-$ instead of $\af_{Z,\pi}^-$ with 
$-\lambda$ being the lowest weight of $(\pi, V)$. 
Let $\pi, \pi'$ be two irreducible representations as above with lowest weights $-\lambda$ and $-\lambda'$,
and with $H$-fixed vectors $v_H$ and $v_H'$.
Then $\pi\otimes \pi'$ has 
a lowest weight vector $v_{-\lambda} \otimes v_{-\lambda'}'$ and
an $H$-spherical vector $v_H \otimes v_H'$.
Let $(\rho, W)$ be the irreducible sub-representation of $\pi \otimes \pi'$ which is generated by the 
lowest weight vector $v_{-\lambda} \otimes v_{-\lambda'}'$. Write $p: V\otimes V' \to W$ 
for the $G$-equivariant projection. Then $0\neq w_H := p(v_H \otimes v_H')$ is an $H$-fixed vector 
of $W$.

\begin{prop} \label{tensor} With the notation introduced above the following holds:
\begin{equation} \label{TENSOR}   \af_{Z, \lambda+\lambda'}^- = \af_{Z, \lambda}^- \cap \af_{Z,\lambda'}^-\, .\end {equation} 
\end{prop} 

\begin{proof} The inclusion $\supseteq $ is clear from Lemma \ref{dualcone} and
the relation $\Lambda_\rho\subseteq\Lambda_\pi+\Lambda_{\pi'}$.  To obtain the opposite inclusion let us call an element 
$\mu \in \Lambda_\pi \cup \Lambda_{\pi'}$ indecomposable if it cannot be expressed as $\mu = \mu_1 +\mu_2$ with 
non-zero elements $\mu_i \in  \Lambda_\pi \cup \Lambda_{\pi'}$.
Thus the opposite inclusion will follow provided 
that $\mu \in  \Lambda_\rho$ for all indecomposable $\mu$.  
Without loss of generality let $\mu \in \Lambda_\pi$. 

Note that the matrix coefficient
$( g \cdot (v_{\lambda}^*\otimes v_{\lambda'}^*) )( v_{-\lambda +\mu} \otimes v_{-\lambda'}')$
is the product of nonzero matrix coefficients $(g\cdot v_\lambda^*)(v_{-\lambda+\mu})$ and 
$(g\cdot v_{\lambda'})(v_{-\lambda'}')$. Hence we can select
$g\in G$ such that 
\begin{equation} \label{TREFFER} 
( g \cdot (v_{\lambda}^*\otimes v_{\lambda'}^*) )( v_{-\lambda +\mu} \otimes v_{-\lambda'}')
\neq 0\, .\end{equation}
We consider the following algebraic function $F$ on $A$:
$$ F(a) =   
(g \cdot (v_{\lambda}^*\otimes v_{\lambda'}^*) ) (a\cdot w_H) = (g \cdot (v_{\lambda}^*\otimes v_{\lambda'}^*) )( a\cdot (v_H \otimes v_H'))\,.$$
Then $F$ is a linear combination of characters on $A$. As $\mu$ was indecomposable, 
(\ref{TREFFER}) implies that the coefficient of $ a^{-\lambda- \lambda' +\mu}$ in the expansion 
of $F$ is non-zero. Hence $\mu\in\Lambda_\rho$ and the proof is
complete. 
\end{proof} 

In order to move on we need:

\begin{lemma}\label{fin gen} 
The multiplicative monoid ${\mathcal P}_{+, \1}$ 
is finitely generated. 
\end{lemma}

\begin{proof} Recall (\ref{P1+}) and observe that 
$$ \R[G/H]^{MN} \simeq \R [G/MN\times Z]^G\,.$$ 
Further $G/MN$ is quasi-affine and hence so is $G/MN\times Z$. As $G$ is algebraic real reductive 
we conclude from Hilbert's theorem that  $\R[G/MN\times Z]^G$ is finitely generated. 
\end{proof}

Let 
$f_1, \ldots, f_n$ be a set of generators of ${\mathcal P}_{+,\1}$
and let $f_i=f_{\pi_i}$. Then
Proposition \ref{tensor} implies that $\af_Z^- = \cap_i \af_{Z,\pi_i}^-$. 
Moreover, we have:

\begin{lemma} \label{lemma INDEPENDENCE}  Let $ f=f_\pi \in {\mathcal P}_{++,\1}$. Then 
\begin{equation} \label{INDEPENDENCE} \af_Z^- = \af_{Z,\pi}^- \, .\end{equation}
\end{lemma}

\begin{proof} As in \cite{BLV}, Lemme 2.1, we use that 
$\C[G]$ is a Krull domain to conclude that the product $f_1\cdots f_n $ divides $f^N$ for some 
$N\in \N$. The result then follows from Proposition \ref{tensor}. 
\end{proof}

So far we have assumed that $Z$ is quasi-affine. The general case is reduced to 
this case as follows. 

One needs to allow the vector $v_H=v_{H,\chi} $ to transform under the
character $\chi$ of $H$.  
Let
$$f_{\pi,\chi}(g) := v_\lambda^* (\pi(g) v_{H,\chi}).$$
As before we expand
$$ v_{H,\chi} = \sum_{\mu \in \Lambda_{\pi, \chi}} v_{-\lambda +\mu}$$
and observe that each $\mu\in\Lambda_{\pi, \chi} $ is trivial on $\af_H$, 
so that $\Lambda_{\pi, \chi} \subset \af_Z^*$.
As in (\ref{open comp cone}) one defines an open cone $\af_{Z, \pi, \chi}^{--}$,
and taking closures one defines 
$$ \af_{Z,\chi}^-= \bigcap_{\pi} \af_{Z,\pi,\chi}^-,\qquad            
\af_Z^- = \bigcap_{\chi } \af_{Z,\chi}^-= \bigcap_{\pi,\chi } \af_{Z,\pi,\chi}^-$$ 
as in Definition \ref{defi comp cone}.

For each $\chi$ we define the quasi-affine space $Z_{1,\chi}$ as in (\ref{reduction quasi-affine}),
then $\af_{Z_{1,\chi}}^-=\af_{Z,\chi}^-\times\R$.
From our discussion of the quasi-affine situation we deduce that the 
first intersection above is finite and that 
\begin{equation}\label{conechi}  \af_{Z,\chi}^-=\af_{Z,\pi,\chi}^- \quad \hbox{if} \quad 
f_{\pi,\chi} \in {\mathcal P}_{++}\, .\end{equation}
It remains to be seen that $\af_{Z,\chi}$ is independent of $\chi$. We employ the notation 
before Proposition \ref{tensor} and note that $f_{\pi,\chi} \cdot f_{\pi',\chi'}= f_{\rho, \chi\chi'}$.
As in Proposition \ref{tensor} we deduce that $\af_{Z,\pi,\chi}^-\cap \af_{Z,\pi',\chi'}^-= 
\af_{Z,\rho, \chi\chi'}^-$. In particular if both $f_{\pi,\chi}$ and  $f_{\pi',\chi'}$ are in $
{\mathcal P}_{++}$, then we obtain  with (\ref{conechi}) that 
$\af_{Z,\chi}^-\cap \af_{Z,\chi'}^- =\af_{Z,\chi\chi'}$. Proceeding as in the quasiaffine case 
we get
\begin{equation} \label{INDEPENDENCE'} \af_Z^- = \af_{Z,\pi,\chi}^- \quad \hbox{if} \quad 
f_{\pi,\chi} \in {\mathcal P}_{++}\, .\end{equation}

\subsection{The edge of the compression cone}

Let $C$ be a closed convex cone in a finite dimensional real vector space $V$. The edge 
of $C$ is the linear  subspace $E(C):= C\cap -C $ of $V$. One calls $C$ {\it pointed} or {\it sharp} 
provided that $E(C)=\{ 0\}$.

\par We deduce from Lemma  \ref{lem hf} the direct sum decomposition 
for the Lie algebra of $N_G(\hf)$:  

\begin{equation} \label{AH} \nf_\gf(\hf) = \hf \oplus  \tilde \af_h \oplus \tilde \mf_h\end{equation} 
with subspaces $\tilde\af_h \subset\af$ and $\tilde \mf_h \subset \mf$.

\begin{lemma}\label{edge lemma} 
{\rm (Edge of the Compression cone)} The following assertions are 
equivalent: 
\begin{enumerate}
\item $\tilde \af_h=\{0\}$. 
\item $N_G(\hf)/H$ is compact. 
\item $\af_Z^-$ is sharp. 
\end{enumerate}
\end{lemma}
\begin{proof} The equivalence of (1) and (2) 
follows from (\ref{AH}). 
\par By the definition of the compression cone we have 
$$ \tilde \af_h +\af_{Z,\pi,  \chi }^{--} = \af_{Z, \pi, \chi}^{--}$$ 
for each $\pi$ and $\chi$,
and hence $\tilde \af_h \subset \af_Z^-$.  
Thus if $\tilde \af_h\neq\{0\}$, then 
$\af_Z^-$ is not sharp, i.e.~(3) implies (1). 
Finally, it follows from Lemma \ref{dualcone} that
the edge $E(\af_Z^-)$ fixes the line $\R v_{H,\chi}$. 
Hence $E(\af_Z^-)\subset \jf$ (see Lemma \ref{J lemma}),
and then  $E(\af_Z^-)\subset \nf_\gf(\hf)$ since
$\jf/\nf_\gf(\hf)$ has compact Lie algebra (see Theorem \ref{J-cpct}). Thus 
(1) implies that $E(\af_Z^-) \subset \hf$ and (3) holds. 
\end{proof}

\subsection{Compression in the Grassmannian} 

We define the {\it limiting subalgebra}  
$$ \hf_{\rm lim} = \oline \uf + (\lf\cap\hf)\, .$$
{}From (\ref{deco2}) we obtain that $\hf_{\rm lim}$ is a real spherical 
subalgebra of $\gf$ with 
$$d:=\dim \hf =\dim \hf_{\rm lim}.$$
Let us denote by $\Gr_d(\gf)$ the Grassmannian of $d$-dimensional 
subspaces of the real vector space $\gf$.  

\par Denote by $\af_Z^{--}$ the interior of $\af_Z^-$.  The following 
result is motivated by the work of Brion, see  \cite{Brion}, Section 2.

\begin{lemma} \label{Grass} Let $X\in \af_Z$. Then the 
following statements are equivalent: 
\begin{enumerate}
\item $X\in \af_Z^{--}$. 
\item $\lim_ {t\to \infty} e^{t\ad X} \hf = 
\hf_{\rm lim}$ in $\Gr_d(\gf)$.
\end{enumerate}
\end{lemma}

\begin{proof} Let $X\in \af_Z$. 
Set $V:=\bigwedge^d \gf$ and endow $V$ with an inner product $\la\cdot, \cdot\ra$
which is $\theta$-covariant, that is: $\la g\cdot v, w\ra = \la v, \theta(g)^{-1}\cdot w\ra$
for all $v, w\in V$ and $g\in G$. 

\par Let $v_0, v_H\in V$ be lifts of $\hf_{\rm lim}, \hf\in \Gr_d(\gf)$
to $V$. Note that $v_0$ is an eigenvector for $A\oline N$ with $\af$-weight 
$-2\rho_\uf:= -\sum_{\alpha\in \Sigma_\uf} \alpha$. 
Denote by $W$ the $G$-submodule of $V$ which is generated by 
$v_0$ and write $p: V\to W$ 
for the orthogonal projection. Set $w_H:=p(v_H)$ and $w_0:=p(v_0)=v_0$. 
Then we need to show that: 
\begin{equation} \label{Grass-compr} X \in \af_Z^{--} \iff 
\lim_{t\to \infty} [\exp(tX)\cdot w_H]=[ w_0]
\, \end{equation}
in $\mathbb{P}(V)$.

\par Let $W^*$ be the dual representation of $W$ and $w_0^*\in W^*$ be a functional with 
$an \cdot w_0^*= a^{2\rho_\uf} w_0^*$ for all $an \in AN$ and $w_0^*(w_0)\neq 0$. 
Attached to $w_0^*$ and $w_H$ is the $AN\times H$-semi-invariant function 
$$ F(g):= w_0^*(g\cdot w_H)\, .$$
In general this is not a left $P$-eigenfunction, but we can overcome this difficulty
by passing to the $M$-average of the square as in the proof of Lemma \ref{J-inv}: 
\begin{equation} \label{constrf} f(g):= \int_M F(m\cdot g)^2 \ dm \qquad (g\in G)\, .\end{equation}
Then $f$ is a matrix coefficient of a finite dimensional irreducible representation 
$U$ with $H$-fixed vector $u_H$ and a $P$-eigenvector $u_0^*\in U^*$ with $man\cdot u_0^*= 
a^{4\rho_\uf} u_0^*$: 
$$ f(g) = u_0^* (g\cdot u_H)\qquad (g\in G)\, .$$
As $4\rho_\uf(\alpha^\vee)>0$ for all $\alpha\in \Sigma_\uf$ we conclude 
that 
$f \in {\mathcal P}_{++}$. 
From (\ref{INDEPENDENCE}) we thus conclude that 
$X\in \af_Z^{--}$ if and only if 
\begin{equation} \label{Grass-compr2}
\lim_{t \to \infty} [\exp(tX) \cdot u_H ]= [u_0] \end{equation}
holds true. 
Hence it remains to be seen that (\ref{Grass-compr2}) is equivalent to 
the right hand side of (\ref{Grass-compr}). 

For that let $w_H= \sum_{\nu\in \Lambda_F}  v_{-2\rho_{\uf}+\nu}$ and 
$u_H=\sum_{\nu\in \Lambda_f} u_{-4\rho_{\uf} +\nu}$ 
be the respective decomposition into non-trivial
$\af$-weight vectors, then $\Lambda_f\subseteq \Lambda_F+\Lambda_F$,
and hence (\ref{Grass-compr2}) is implied by
the right hand side of (\ref{Grass-compr}). For the converse implication,
we observe that if $\nu\in\Lambda_F$ is indecomposable 
(that is, not the sum of two non-zero elements from $\Lambda_F$),
then $2\nu\in\Lambda_f$. This follows from the construction (\ref{constrf}),
as in the proof of Proposition \ref{tensor}.
The implication is an easy consequence of this observation. 
\end{proof}

\subsection{An explicit description of the compression cone}

This part is an adaption of Section 2.3 from \cite{Brion}. 

Write $\df_H^\perp$ for the orthocomplement of $\df_H:= \df \cap \hf$ in $\df$. 
The local structure theorem implies that 
$$ \hf = \lf \cap \hf \oplus {\mathcal G}(T)$$
where ${\mathcal G}(T)$ is the graph of a linear map 
$$ T: \oline{\uf}\to \uf + \df_H^\perp\, .$$ 

In particular, for all $\alpha\in \Sigma_\uf$ and root vectors $X_{-\alpha}\in \gf^{-\alpha}$
we find $D_\alpha\in \df_H^\perp$  and $X_\beta \in \gf^\beta$, $\beta \in \Sigma_\uf$ 
such that 

$$ Y_\alpha:=X_{-\alpha} + D_\alpha + \sum_{\beta} X_\beta\in \hf\, .$$
As $[\af_H, Y_\alpha]\in {\mathcal G}(T)$ we conclude that:
$\alpha|_{\af_H}=0$ if $D_\alpha\neq 0$, and $(\alpha+\beta)|_{\af_H}=0$ 
if $X_\beta\neq 0$. 
We let ${\mathcal M}\subset \af_Z^*\simeq \af_H^\perp \subset \af^*$ be the monoid , i.e. additive semi-group, 
 generated by 

\begin{itemize} 
\item $\alpha$ if there exists $X_{-\alpha}$ with $D_\alpha\neq 0$.  
\item $\alpha+\beta$ if there exists $X_{-\alpha}$ with $X_\beta\neq 0$.  
\end{itemize}

We combine Lemma \ref{dualcone} with Lemma \ref{Grass} and obtain
that: 

\begin{lemma}\label{explicitcone} $\af_Z^-=\{ X\in \af_Z\mid (\forall \alpha \in {\mathcal M})
\ \alpha(X)\leq 0\}$. 
\end{lemma}

\begin{rmk} In the case where $Z_\C$ is spherical, the 
compression cone coincides with the so-called valuation cone. 
This follows from Lemma \ref{explicitcone}
in combination with \cite{Brion}, Cor.~2.4.   
\end{rmk}

\subsection{Refined Polar decomposition}

\par Set $A_Z^-:=\exp{\af_Z^-}\subseteq A/A_H$. 

\begin{prop} \label{wavefront}  Let $Z=G/H$ be  a real spherical space with 
$H=J$ and simple compactification $\hat Z$ as in Section \ref{pode}. 
Then the map 
$$\Phi: K\times \overline{A_Z^{-}F\cdot z_0} \to \hat Z, \ \ (k,s) \mapsto k\cdot s$$
is open in $(\1,\hat z)$. 
In particular one has $G= F' K A_Z^- FH$ for a finite set $F'\subset G$. 
\end{prop}

\begin{proof}  Let $(\pi, V)$ be an $H$-semi-spherical representation out of which 
we constructed $\hat Z \subset{\mathbb P} (V)$. 
We use that $\af_Z^-$ is the compression cone of $\pi$, see Lemma \ref{lemma INDEPENDENCE}
and the discussion below it.
Note that  $z_0 = [v_H]$ and 
$\hat z = [v_0]$. We claim that  $\overline{MA_Z^{-}F\cdot z_0}$ is a neighborhood of $\hat z$ in the slice. 
In fact, by the definition of the compression cone we see that $ T_Z (A_Z\bs A_Z^-)\cdot z_0$ does not 
meet a neighborhood of $\hat z$. The claim follows. 
\par By Lemma \ref {lemma2} the map $\Phi$  is open and 
the assertion follows as in the proof of Lemma \ref{lemma1}.
\end{proof}

As in Section \ref{pode} we obtain as a corollary: 

\begin{theorem} {\rm(Fine Polar decomposition)} \label{generic RPD}  
Suppose that $Z=G/H$ is a real spherical space.  
Then there exists a compact subset $\Omega\subset G$ and a finite set $F''\subset G$ 
such that 
$$G= \Omega A_Z^-  F'' H\,  .$$
Moreover, one has:
\begin{enumerate}
\item $\Omega= F' K$ for some finite set $F'\subset G$.
\item $F''\subseteq FN_G(H)$.
\end{enumerate}
\end{theorem}

\begin{ex} In many examples it appears that $F''$ can be taken in $N_G(H)$. 
However it cannot be skipped completely as 
the basic example $Z= \Sl(2,\R)/\SO(1,1)$ shows. Let $A$ be the diagonal 
matrices in $G=\Sl(2,\R)$ with positive entries and $K=\SO(2,\R)$.  
We realize $Z$ as the one sheeted hyperboloid in $\R^3$:
$$ Z= \{ (x,y,z)\in \R^3 \mid x^2 + y^2 -z^2 =1\}$$ 
with base point $z_0=(1,0,0)$. Then 
$$A_Z^-\cdot z_0=A^-\cdot z_0 =\{ (x,0,z)\in Z \mid x,z\geq 0\}\,.$$
Hence it it is not possible that $Z=\Omega A_Z^-\cdot z_0$ holds for 
a compact set $\Omega\subset G$. 
Here $N_G(H)_0 =H$ and the quotient $N_G(H)/H$ is realized by the involutive 
element $w=\begin{pmatrix} 0 &1 \\ -1 & 0 \end{pmatrix} 
\in K$ which satisfies $wA_Z^-w= A_Z^+$. 
With $\Omega=K$ one then has $G=KA_ZH$. 
\end{ex}

\section{Wavefront spherical spaces and the wavefront lemma}

Denote by $\af^-$ the closure of the negative Weyl chamber $\af^{--}$.  
The following definition is motivated by  \cite{SV}. 

\begin{definition}\label{defi wavefront}  We call the real spherical space $Z=G/H$ 
{\it wavefront} provided that 
\begin{equation} \label{wf} \af_Z^{-} = (\af^- +\af_H)/\af_H\, .\end{equation}                  
\end{definition}

\begin{rmk} All symmetric spaces are wavefront. Moreover, if $G$ and $H$ are complex then one can 
decide with  the Luna diagram whether $Z=G/H$ is wavefront. For example all 
complex spherical spaces of the type $Z=G\times H / \diag (H)$ are wavefront.  A few others, 
such as $\SO(2n+1,\C)/ \GL(n,\C)$,  $\GL(2n+1, \C)/ {\rm Sp}(2n, \C)$ or $\SO(8,\C)/ G_2$ are 
not wavefront. To be precise:  from the 78 cases in the list of \cite{BraPez}, the non-wavefront 
cases are: (11), (24), (25), (27), (39-50), (60) and (61).
\end{rmk}

Denote by $A^-$ the closure of the negative Weyl chamber with respect to the positive system determined by $N$. 
Notice that (\ref{wf}) implies that 
\begin{equation}\label{AZ}
A^-_Z\cdot z_0 = A^-\cdot z_0.
\end{equation}

With that we obtain a generalization of 
 the ``wavefront lemma'' of Eskin-McMullen 
(\cite{EM} Theorem 3.1). The technique of the proof is essentially known -- see for instance
\cite{SV} or \cite{KSS}.   

\begin{lemma}\label{wfl} Suppose that $Z=G/H$ is a wavefront real spherical space. 
Then there exists a closed subset $E\subset G$  with the following properties.
\begin{enumerate}
\item $E\to G/H$ is surjective.
\item For  every  neighborhood ${\mathcal V}$ of $\1$ in $G$, there exists 
a neighborhood ${\mathcal U}$ of $\1$ in $G$ such that
$$ {\mathcal V}g\cdot z_0 \supset g{\mathcal U}\cdot z_0$$
for all $g\in E$.
\end{enumerate}
\end{lemma}

\begin{proof} Put  $$E=\Omega A^-F''.$$
Then (1) follows from (\ref{AZ}) and
Corollary \ref{generic RPD}.

The proof of (2) is similar to \cite{KSS}, Lemma 5.4. For the convenience to the 
reader we recall the argument. 

\par For a compact set $\Omega\subset G$ we note that the set $\bigcap_{x\in \Omega} \Ad(x^{-1}) \mathcal{V}$
is a neighborhood of ${\bf 1}$ in $G$. Then the assertion is 
reduced to the case 
where $g\in A^-F''$.  

\par Let ${\mathcal U}_1$ be a neighborhood of ${\bf 1}$ in $P$
which is contained in $\mathcal {V}$ and which is
stable under conjugation 
by elements from $A^-$.
As $P fH$ is open for all $f \in F''$, we see that $f^{-1} {\mathcal U}_1f \cdot z_0$ is a
neighborhood of $z_0$ for each $f\in F''$.
We choose ${\mathcal U}$ so small that 
$${\mathcal U}\cdot z_0\subset f^{-1} {\mathcal U}_1 f \cdot z_0$$
holds for all $f\in F''$. 
Then for $g=af$ with $a\in A^-$
we obtain 
$$af{\mathcal U}\cdot z_0 \subset a  {\mathcal U}_1 f \cdot z_0
\subset {\mathcal U}_1 af \cdot z_0\subset {\mathcal V}af \cdot z_0\,.$$
\end{proof}

\begin{rmk} Lemma \ref{wfl} suggests that real spherical spaces  
which are wavefront are especially suited to discuss the lattice counting problem, 
see \cite{EM}. Having developed the harmonic analysis on real spherical spaces 
further one can obtain error term bounds for the lattice counting problem. 
We will return to this topic  in an upcoming publication. 
\end{rmk}

\end{document}